\documentclass[a4paper]{amsart}
\usepackage{amscd,amsthm,amsfonts,amssymb,amsmath,esint}
\usepackage{amsmath,tikz-cd}
\usepackage[colorlinks=true]{hyperref}
\usepackage{fullpage}
\usepackage[all]{xy}
\usepackage{mathtools}
\usepackage{color}

    \newcommand{\BC}{{\mathbb {C}}} 
     
    \newcommand{\BG}{{\mathbb {G}}}

     \newcommand{\BR}{{\mathbb {R}}}

     \newcommand{\BZ}{{\mathbb {Z}}}

    \newcommand{\RM}{{\mathrm {M}}} 
    \newcommand{\RO}{{\mathrm {O}}} 
     
    \newcommand{\RS}{{\mathrm {S}}} 
    \newcommand{\RU}{{\mathrm {U}}}

    \newcommand{\fg}{{\mathfrak{g}}}

     \newcommand{\ft}{{\mathfrak{t}}}

    \newcommand{\diag}{{\mathrm{diag}}}

    \newcommand{\ed}{\mathrm{ed}}

    \newcommand{\Flag}{\mathrm{Flag}}
    \newcommand{\CFlag}{\mathrm{CFlag}}

     \newcommand{\GL}{{\mathrm{GL}}}

    \newcommand{\Ker}{{\mathrm{Ker}}}

    \newcommand{\SL}{{\mathrm{SL}}}
     
    \newcommand{\SO}{{\mathrm{SO}}}
    \newcommand{\Sp}{{\mathrm{Sp}}}
    \newcommand{\SU}{{\mathrm{SU}}}
    \newcommand{\Sym}{{\mathrm{Sym}}}
    \newcommand{\Stab}{{\mathrm{Stab}}}

    \newcommand{\tr}{{\mathrm{tr}}}

    \newcommand{\vol}{{\mathrm{vol}}}

    \newcommand{\incl}{\hookrightarrow}

    \newcommand{\bs}{\backslash}

    \theoremstyle{plain}
    \newtheorem{thm}{Theorem}[section] \newtheorem{cor}[thm]{Corollary}
    \newtheorem{lem}[thm]{Lemma}
    
        \newtheorem{ques}[thm]{Question} \newtheorem{prop}[thm]{Proposition}
    
    \newtheorem{defn}[thm]{Definition}
     \newtheorem{lem-defn}[thm]{Lemma-Definition}

\theoremstyle{remark} \newtheorem{remark}[thm]{Remark}
\theoremstyle{remark} 
\theoremstyle{remark} \newtheorem{example}{Example}

    \numberwithin{equation}{section}

\begin{document}
\title{Minimal dimension equivariant embeddings of real and complex flag manifolds into Euclidean spaces}
\author{Zhongzi Wang, and Hang Yin}
\begin{abstract}
    We determine the minimal equivariant embedding dimension of orthgonal groups acting on real flag manifolds and unitary groups acting on complex flag manifolds. The minimal embedding dimension is achieved at isospectral model.
\end{abstract}
\maketitle

\section{Introduction}


A famous theorem in differential topology is any smooth $n$-dimensional manifold can be smoothly embedded into $\mathbb{R}^{2n}$ \cite{Wh44}.  There are many results concerning the embedding of manifolds see e.g. Hirsch \cite{Hi61} proved every orientable $3$-manifold smoothly embedded into $\mathbb{R}^5$, and actually the dimension $5$ is the best possible for general 3-manifolds.  
The determination for the minimal dimension of smooth embedding of a specific manifold is in general a very hard problem, 
and there are many results on the estimate of embedding dimensions.

We first recall the definition of equivariant embedding:

\begin{defn} Suppose $M$ is a smooth $n$-dimensional manifold and $G$ be a Lie group acting smoothly on $M$. A $G$-equivariant embedding $e$ is a smooth embedding $e: M\to \mathbb{R}^N$ for some positive integer $N$ together with a Lie group representation $\rho: G\to {\rm O}(N)$, such that for any $g\in G$, $\rho(g)\circ e=e\circ g$.
\end{defn}
Suppose $M$ is a smooth compact manifold with a Lie group $G$-action. Mostow \cite{Mos57}, also Palais \cite{Pal57}, proved $M$ has an $G$-equivariant embedding into $\mathbb{R}^{N}$ for some positive integer $N$. 
We next recall the definition of a homogeneous manifold:

\begin{defn} Suppose $M$ is a smooth compact $n$-dimensional manifold and $G$ be a compact Lie group acting smoothly on $M$. We say $M$ is a homogeneous manifold if the $G$-action on $M$ is transitive.
\end{defn}

The general problem is

\begin{ques}
Suppose $M$ is a smooth compact $n$-dimensional manifold and $G$ be a compact Lie group acting smoothly on $M$. If $M$ is a homogeneous manifold, let $\text{ed}_G(M)$ be the minimal $N$ such that there exists a $G$-equivariant embedding from $M$ to $\mathbb{R}^N$. Determine the value of $\text{ed}_G(M)$.
\end{ques}

There are results concerning the upper bound for the dimension of equivariant embedding of manifolds see e.g. \cite{Was69}. When $G$ is a finite group, see also \cite{Wan22}. 

There has been studies on minimal dimension of  equivariant embedding of finite group action on surfaces, see  \cite{WWZZ}, \cite{WWW}. Recently, \cite{LY} determined the minimal dimension of equivariant embedding real flag manifolds with orthogonal group action except finitely many cases. In this paper, we determined the minimal dimension for equivariant embedding for arbitrary real flag manifolds and complex flag manifolds.



The core method for studying $\text{ed}_G(M)$ will be the use of the representation theory of Lie groups. Take a basepoint $x_0\in M$. Then the point stabilizer $\Stab_{x_0}(G)=\{g\in G\mid g\cdot x_0=x_0\}$ is a closed subgroup of $G$, and we have an diffeomorphism $M\cong G/\Stab_{x_0}(G)$. 

The problem of equivariant embedding is equivalent to the problem of finding a representation space of $G$ that contains a point with stabilizer equals to $\Stab_{x_0}(G)$. For representations of Lie groups, see \cite{FH91}.

For an $n\times n$-matrix $T$, denote $T^t$ the trasnpose of $T$.
If $T\in M_n(\mathbb{C})$, denote $T^*$ the conjugate transpose of $T$.
\subsection{Isospectral model over $\mathbb{R}$}

Let $\text{Sym}_n(\mathbb{R})$ denote the space of symmetric $n\times n$-matrices over $\mathbb{R}$. Let $\text{SO}_n(\mathbb{R})$ denote the special orthogonal group. Note that $\text{SO}_n(\mathbb{R})$ acts on $\text{Sym}_n(\mathbb{R})$ by 
$A\cdot M=AMA^{t}$. So $\Sym_n(\mathbb{R})$ is a representation space of the special orthogonal group $\text{SO}_n(\mathbb{R})$. The irreducible decomposition of this representation space is
$$\text{Sym}_n(\mathbb{R})=\text{Sym}^{tr=0}_{n}(\mathbb{R})\oplus \mathbb{R}$$
where $\text{Sym}_n^{tr=0}(\mathbb{R})=\{T\in \text{Sym}_n(\mathbb{R}): trT=0\}.$
We say the representation space $\text{Sym}_n(\mathbb{R})$ is the isospectral model and the representation space $\text{Sym}_n^{tr=0}(\mathbb{R})$ is the traceless isospectral model.
\subsection{Isospectral model over $\mathbb{C}$}
Let $\text{Her}_n(\mathbb{R})$ denote the space of Hermitian $n\times n$ matrices, i. e. $\text{Her}_n(\mathbb{R})=\{T\in M_n(\mathbb{C}): T=T^*\}$. Let $\text{SU}_n$ be the special unitary group.
Note that $\text{SU}_n$ acts on $\text{Her}_n(\mathbb{R})$ by 
$A\cdot M=AMA^{*}$, where $A^*=\overline{A^t}$. So $\text{Her}_n(\mathbb{R})$ is a representation space of the special unitary group $\text{SU}_n$. The irreducible decomposition of this representation space is
$$\text{Her}_n(\mathbb{R})=\text{Her}^{tr=0}_{n}(\mathbb{R})\oplus \mathbb{R}$$
where $\text{Her}_n^{tr=0}(\mathbb{R})=\{T\in \text{Her}_n(\mathbb{R}): trT=0\}.$
We say the representation space $\text{Her}_n(\mathbb{R})$ is the isospectral model and the representation space $\text{Her}_n^{tr=0}(\mathbb{R})$ is the traceless isospectral model.

\subsection{Real and complex flag manifolds}
The real flag manifold $\Flag(n_1,...,n_r,\mathbb{C}^n)$ is a smooth manifold 
with a transitive $\text{SO}_n(\mathbb{R})$-action. 
There is an $\SO_n(\mathbb{R})$-equivariant embedding from $\Flag(n_1,...,n_r)$ into the traceless isospectral model $\Sym_n^{tr=0}(\mathbb{R}).$
Let $A$ be the block diagonal matrix of diagonal entries $a_1I_{k_1},...,a_rI_{k_r}$ where $a_1,...,a_r$ are distinct real numbers with $\sum_{i=1}^ra_ik_i=0$. The image of the $\SO_n(\mathbb{R})$-equivariant embedding of the flag manifold $\Flag(k_1,...,k_r;\mathbb{R}^n)$ is the $\SO_n$-orbit of $A$ in $\Sym_n^{tr=0}(\mathbb{R}).$

The complex flag manifold $\Flag(n_1,...,n_r,\mathbb{C}^n)$ is a smooth complex manifold 
with a transitive $\SU_n$-action. 
There is an $\SU_n$-equivariant embedding from $\Flag(n_1,...,n_r)$ into the traceless isospectral model $\text{Her}_n^{tr=0}(\mathbb{R}).$ Let $A$ be the same matrix as in the real case.
The image of the $\SU_n$-equivariant embedding of the flag manifold $\CFlag(k_1,...,k_r;\mathbb{C}^n)$ is the $SU_n$-orbit of $A$ in $\text{Her}_n^{tr=0}(\mathbb{R}).$

The equivariant embedding for the real flag manifold is due to Lim-Ye. 
\subsection{Basic of flag variety}
\begin{defn}\label{real flag}
	Let $\mathbb{K}$ be a field. Let $V$ be an $n$-dimensional vector space over $\mathbb{K}$. Let $0<k_1<...<k_r<n$. 
    \begin{enumerate}
    \item A flag of type $(k_1,...,k_r)$ in $V$ is a filtration of linear subspaces $W_1\subset ... \subset W_r\subset V$ such that $\text{dim}W_i=k_i.$ 
    \item When $k_i=i$, the flag is said to be a \textbf{full flag}.

    \item Given $0<k_1<...<k_r<n$, the set of all flags of type $(k_1,...,k_r)$ in $V$ can be equipped with the structure of a smooth projective algebraic variety over $\mathbb{K}$. It is called the flag variety of type $(k_1,...,k_r)$ of $V$, and is denoted as $\text{Flag}(k_1,...,k_r, V)$. 
    \item When $\mathbb{K}=\mathbb{R}$ or $\mathbb{C}$, we also call it a flag manifold of type $(k_1,...,k_r)$.
    \end{enumerate}
\end{defn}

Now we assume $\mathbb{K}=\mathbb{R}$ or $\mathbb{C}$. The flag manifold $\Flag(k_1,...,k_r, V)$ 
has dimension $$\frac{n(n-1)}2-\sum_{i=1}^{r+1}\frac{l_i(l_i-1)}2$$ 
where $l_1=k_1, l_i=k_i-k_{i-1}$ for $2\le i \le r$ and $l_{r+1}=n-k_r$.
When $r=1$ and $k_1=k$, $\Flag(k_1,..., k_r)$ is the real Grassmann manifold $Gr_k(\mathbb{R}^n)$. The dimension of $Gr_k(\mathbb{R}^n)$ is $k(n-k)$.

It is known that for any inner product on $V$, $\SO(V)$ acts transitively on $\Flag(k_1,\ldots,k_r,V)$, and the stabilizer in $O(V)$ of a flag $W_1\subset\cdots\subset W_r$ in $\Flag(k_1,\ldots,k_r,V)$ is $\RO(U_1)\times\cdots\times\RO(U_{r+1})$. Here $U_i$ is the orthogonal complement of $W_{i-1}$ in $W_i$, and we identify $W_0=0$, $W_{r+1}=V$. If we take stabilizer in $\SO(V)$, then we get $\RS(\RO(U_1)\times\cdots\times\RO(U_{r+1}))$. Here the outer $S$ means the total determinant is $1$. 

Therefore we have $\text{Flag}(k_1,...,k_r)\cong \text{SO}_n(\mathbb{R})/H$ where $H=\text{S}(\text{O}_{l_1}(\mathbb{R})\times \text{O}_{l_2}(\mathbb{R})\times...\times \text{O}_{l_{r+1}}(\mathbb{R}))$.
\subsection{Main theorem and preliminary discussions}

We are considering equivariant embedding of $\SO_n(\BR)/\RS(\RO_{n_1}(\BR)\times\cdots\RO_{n_r}(\BR))$, $n_1+\cdots+n_r=n$, $n_i\geqslant 1$, $n\geqslant 2$. This amounts to a finite dimensional real representation $V$ of $\SO_n(\BR)$ with a vector $v\in V$ whose stabilizer is $\RS(\RO_{n_1}(\BR)\times\cdots\RO_{n_r}(\BR))$.

The isospectral model $\{A\in\RM_n(\BR): A\text{ symmetric and } \tr\ A=0\}$ is an equivariant embedding whose dimension is $(n+2)(n-1)/2$. 
Our first main theorem is the following.

\begin{thm}\label{orth-flag-min}
    The equivariant embedding of $\SO_n(\BR)/\RS(\RO_{n_1}(\BR)\times\cdots\RO_{n_r}(\BR))$, $n_1+\cdots+n_r=n$, $n_i\geqslant 1$, $n\geqslant 2$ with minimal dimension is isospectral model.
\end{thm}

Especially, the minimal dimension for $\text{SO}_n(\mathbb{R})$-equivariant embedding of the real Grassmann manifold $\text{Gr}_k(\mathbb{R}^n)$ is $(n+2)(n-1)/2$ and is independent of $k$.

Lim-Ye \cite{LY} proved the case $n\ge 17$ of Theorem 1.7.

The definition of real flag manifold can be natuarally extended to the complex case, that is, replacing all the real vector spaces by the complex one. In this section, we always assume that $V$ is a complex vector space and denote the corresponding complex flag manifold of type $(k_1,\dots,k_r)$ by $\CFlag(k_1,\dots,k_r,V)$

It is known that for any Hermitian inner product on $V$, $\SU(V)$ acts transitively on $\CFlag(k_1,\ldots,k_r,V)$, and the stabilizer in $U(V)$ of a flag $W_1\subset\cdots\subset W_r$ in $\CFlag(k_1,\ldots,k_r,V)$ is $\RU(U_1)\times\cdots\times\RU(U_{r+1})$. Here $U_i$ is the orthogonal complement of $W_{i-1}$ in $W_i$, and we identify $W_0=0$, $W_{r+1}=V$. If we take stabilizer in $\SU(V)$, then we get $\RS(\RU(U_1)\times\cdots\times\RU(U_{r+1}))$. Here the outer $S$ means the total determinant is $1$.

Our second main theorem is:
\begin{thm}\label{uni-min-equi}
    The equivariant embedding of $\SU_n/\RS(\RU_{n_1}\times\cdots\times\RU_{n_r})$, $n_1+\cdots+n_r=n$, $n_i\geqslant 1$, $n\geqslant 2$ with minimal dimension is isospectral model.
\end{thm}
To sum up, we proved that
\begin{thm}
The minimal dimension equivariant embedding of the flag manifold $\Flag(k_1,...,k_r,\mathbb{R}^n)$ with $\SO_n(\mathbb{R})$-action (resp. $\CFlag(k_1,...,k_r,\mathbb{C}^n))$ with $\SU_n$-action) is the given by the equivariant embedding to the traceless isospectral model $\Sym_n^{tr=0}(\mathbb{R})$ which has dimension $(n-2)(n+1)/2$ (resp. ${\rm Her}_n^{tr=0}(\mathbb{R})$ which has dimension $n^2-1$).
\end{thm}

\subsection{Method to prove main theorems}

(1) The key to prove Theorem \ref{orth-flag-min} is Proposition \ref{orth-flag2}, which states that the diagonal subgroup $H=\text{S}(\text{O}_1(\mathbb{R})\times...\times \text{O}_1(\mathbb{R}))$ acts without nontrivial fixed points on any representation space $V$ of the special orthogonal group $\SO_n(\mathbb{R})$ with $\dim V\le \dim\Sym_n^{tr=0}(\mathbb{R})$ and $V\ne\Sym_n^{tr=0}(\mathbb{R})$.
Here the trivial fixed point in a representation means the origin of that representation space.

(2) The key to prove Theorem \ref{uni-min-equi} is Proposition \ref{H-fix-min}, which states that the diagonal subgroup $H=\RS(\RU_1\times...\times \RU_1)$ acts without nontrivial fixed points on the representation spaces $V$ of the special orthogonal group $\SU_n$ with $\dim V\le\dim \mathrm{Her}_n^{tr=0}(\BR)$ and $V\ne\mathrm{Her}_n^{tr=0}(\BR)$.

Both Proposition \ref{orth-flag2} and \ref{H-fix-min} concerns nontrivial fixed points of a subgroup in a representation space.

{\bf Acknowledgement.} We thank Yu Wang for helpful discussion of the problem and the proof. During the preparation of the paper, the authors found that Lim-Wang-Ye \cite{WLY} obtained the same result independently.

\section{Equivariant embedding of real flag manifolds}

The goal of this section is to prove Theorem \ref{orth-flag-min}

\noindent\textbf{Theorem \ref{orth-flag-min}.}{\it \,
    The equivariant embedding of $\SO_n(\BR)/\RS(\RO_{n_1}(\BR)\times\cdots\RO_{n_r}(\BR))$, $n_1+\cdots+n_r=n$, $n_i\geqslant 1$, $n\geqslant 2$ with minimal dimension is isospectral model.}

\begin{lem}\label{orth-flag1}
    Theorem \ref{orth-flag-min} is true when $n=2$.
\end{lem}
\begin{proof}
    We are considering the case $\SO_2(\BR)/\RS(\RO_1(\BR)\times\RO_1(\BR))\cong S^1/\pm 1$, where $S^1$ is identified with norm $1$ element in $\BC$. The isospectral model corresponds to the representation $S^1\to\GL_1(\BC)$ given by $z\mapsto z^2$.

    We know $S^1$ representation $V$ is direct sum of complex one-dimensional representations. To reach the minimal dimension, $V$ should be complex one-dimensional, hence is given by $S^1\to\GL_1(\BC)$, $z\mapsto z^m$, $m\in\BZ$. The correct one must be $m=\pm 2$. Note that the complex conjugation on $\BC$ transforms the representation $z\mapsto z^{-2}$ to $z\mapsto z^2$, hence $V$ is isomorphic to isospectral model.
\end{proof}

When $n\geqslant 3$, $\SO_n(\BR)$ is a semi-simple Lie group. It is simple when $n\neq 4$.

Denote $H\subseteq\SO_n(\BR)$ the subgroup of diagonal matrices, which should have diagonal entries $\pm 1$. It is clear that for any partition $(n_1,\dots,n_r)$ of $n$, $H \subseteq \RS(\RO_{n_1}(\BR)\times\cdots\RO_{n_r}(\BR))$.

\begin{lem}\label{min-dim-fix}
    Consider the set $\Omega$ of all real representations of $\SO_n(\BR)$ with stabilizer $\SO(\RO_{n_1}(\BR)\times\cdots\times\RO_{n_r}(\BR))$. If $V\in\Omega$ satisfies $\dim V=\min\{\dim W\mid W\in\Omega\}$, then $V$ does not contain trivial sub-representations and each irreducible sub-representation contains a non-zero vector fixed by $H$.
\end{lem}
\begin{proof}
    Suppose $V$ contains a trivial sub-representation, say $\BR\subseteq V$. Then there is a complement sub-representation $W$ such that $V=\BR\oplus W$. Then $W\in\Omega$ with smaller dimension, a contradiction.
    
    Let $W$ be an irreducible sub-representation of $V$ and suppose $H$ does not fix non-zero vector in $W$. Let $W'$ be a complement such that $V=W\oplus W'$. Suppose $v\in V$ has stabilizer $\SO(\RO_{n_1}(\BR)\times\cdots\times\RO_{n_r}(\BR))$. Write $v=w+w'$ with $w\in W$ and $w'\in W'$. Then $H$ fixes $v$ implies $H$ fixes $w$, hence $w=0$. Thus $v=w'$ and $W'\in\Omega$ with smaller dimension, a contradiction.
\end{proof}

The above Lemma emphasizes the importance of real irreducible representations with non-zero $H$-fixed vector. We shall prove the following.

\begin{prop}\label{orth-flag2}
    Assume $n\geqslant 3$ and $n\neq 4$. Let $V$ be an irreducible real representation of $\SO_n(\BR)$. If $\dim V$ is no greater than isospectral model and $H$ fixes some non-zero vector in $V$, then $V$ is either trivial representation or isospectral model.
\end{prop}

\begin{remark}
    The above Proposition is wrong for $\SO_4(\BR)$. For example, the highest weight representation $V(2,2)$ has smaller dimension but $H$ fixes non-zero vector.
\end{remark}

Let us prove main theorem when $n\neq 4$ assuming Proposition \ref{orth-flag2}.

\begin{proof}[Proof of Theorem \ref{orth-flag-min} for $n\neq 4$ assuming Proposition \ref{orth-flag2}]
    By Lemma \ref{orth-flag1}, we may assume $n\geqslant 3$. Assume $V$ is an equivariant embedding of smallest dimension. By Lemma \ref{min-dim-fix}, $V$ does not contain trivial sub-representations and each irreducible sub-representation, say $W$, has non-zero $H$-fixed vector. Now $\dim W$ is no greater than isospectral model, hence Proposition \ref{orth-flag2} implies $W$ should be isospectral model. Considering dimension, we have $V=W$ is isospectral model.
\end{proof}

The remaining case $\SO_4(\BR)$ is proved by hand, see Corollary \ref{SO4-thm}.

The basic ideal is to compute dimension of $V$. However, we can rule out some cases at the begining.

\begin{lem}\label{natural representation-no-fix}
    Let $V=\BR^n$ be the standard representation of $\SO_n(\BR)$. Then $H$ does not fix non-zero vector in $V$.
\end{lem}
\begin{proof}
    This is trivial.
\end{proof}

\begin{lem}\label{adj-no-fix}
    Assume $n\geqslant 3$. Let $V=\mathfrak{so}_n(\BR)$ be the ajoint representation of $\SO_n(\BR)$. Then $H$ does not fix non-zero vector in $V$.
\end{lem}
\begin{proof}
    Suppose $0\neq X\in \mathfrak{so}_n(\BR)$ is fixed by $H$. Suppose the $(i,j)$-entry $X_{ij}$ of $X$ is non-zero, $i<j$. Choose an index $k$ different from $i,j$ (possible since $n\geqslant 3$). Consider diagonal matrix $A\in H$ whose $(i,i)$-entry and $(k,k)$-entry are $-1$ and other diagonal entries are $1$. Then the $(i,j)$-entry of $AXA^{-1}$ is $-X_{ij}$, hence $X_{ij}=0$, a contradiction.
\end{proof}







To classify complex irreducible representations of $SO_n(\mathbb{R})$, we first need to consider the type of the Lie group $SO_n(\mathbb{R})$.

The Lie algebra of $\SO_n(\BR)$ belongs to different types depending on parity of $n$. Thus we distinguish these two cases.

\subsection{Case of $\SO_{2n}$ ($n\geqslant 3$)}

This is type $D$ Lie group.

Each real representation $V$ of $\SO_{2n}(\BR)$ gives rise to complex representation $V_\BC$ of $\SO_{2n}(\BC)$. We shall first compute dimension of irreducible complex representations of $\SO_{2n}(\BC)$.

Denote $I_n$ the $(n\times n)$-identity matrix,
\[S=\frac{1}{\sqrt 2}\begin{pmatrix}I_n&-iI_n\\I_n&iI_n\end{pmatrix}.\]
For matrix $J=\begin{pmatrix}0&I_n\\I_n&0\end{pmatrix}$, denote
\[\SO_{2n,J}(\BC)=\{A\in\GL_{2n}(\BC):A^tJA=J,\det A=1\}.\]
It is easy to verify that
\begin{equation}\label{isom1}
    f:\SO_{2n}(\BC)\to\SO_{2n,J}(\BC),\quad A\mapsto SAS^{-1}
\end{equation}
is an isomorphism of algebraic groups.

The maximal torus $T$ of $\SO_{2n,J}(\BC)$ can be taken as $\{\diag(a_1,\ldots,a_n,a_1^{-1},\ldots,a_n^{-1}):a_1,\ldots,a_n\in\BC^\times\}$. The charaters of $T$ are identified with $\BZ^n$. We shall denote such a character by $\mu=(\mu_1,\ldots,\mu_n)$, $\mu_i\in\BZ$. Then
\[\mu(\diag(a_1,\ldots,a_n,a_1^{-1},\ldots,a_n^{-1}))=\prod_i a_i^{\mu_i}.\]
Denote $e_i\in\BZ^n$ with only non-zero value $1$ at coordinate $i$. Choose positive roots to be $e_i\pm e_j$, $(i<j)$. It turns out that highest weight complex representations of $\SO_{2n,J}(\BC)$ correspond to dominant weights $\mu\in\BZ^n$ with $\mu_1\geqslant\cdots\geqslant\mu_{n-1}\geqslant|\mu_n|$. Denote by $V(\mu)$ the highest weight $\mu$ representation.

It is easy to compute the dimension by Weyl character formula,
\begin{equation}\label{orth-even-dim}
    \dim V(\mu)=\prod_{i<j}\frac{\mu_i-\mu_j+j-i}{j-i}\prod_{i<j}\frac{\mu_i+\mu_j+2n-i-j}{2n-i-j}.
\end{equation}

Denote $E(\mu)$ the enumerator in the formula (\ref{orth-even-dim}) of $\dim V(\mu)$.

\begin{lem}
    For a dominant weight $\mu\in\BZ^n$, debite $\mu'=(\mu_1,\ldots,\mu_{n-1},-\mu_n)$. Then $\dim V(\mu)=\dim V(\mu')$.
\end{lem}
\begin{proof}
    By the dimension formula, only need to compare the terms in enumerator involving index $n$. In the formula of $\dim V(\mu)$, the terms are
    \[\prod_{i<n}(\mu_i-\mu_n+n-i)\prod_{i<n}(\mu_i+\mu_n+n-i).\]
    This expression is unchanged when $\mu_n$ becomes $-\mu_n$.
\end{proof}

In computing dimension, we can only consider dominant weights with non-negative coordinates.

\begin{lem}\label{first type}
	When $k\geqslant 3$ and $n\geqslant 4$, we have
	\[\dim V(\underbrace{1,\ldots,1}_k,0,\ldots,0)\geqslant\dim V(2,0,\ldots,0).\]
        The equality holds only at $\dim V(1,1,1,1)=\dim V(2,0,0,0)$.
\end{lem}
\begin{proof}
	Denote $\mu=(2,0,\ldots,0)$ and $\lambda_k=(\underbrace{1,\ldots,1}_k,0,\ldots,0)$. We prove that when $k\leqslant n-2$, we have $\dim V(\lambda_{k+1})>\dim V(\lambda_k)$. Only need to compare terms in $E(\lambda)$ involving index $k+1$. The terms in $E(\lambda_k)$ are
	\[\prod_{i\leqslant k}(k+2-i)\prod_{j\geqslant k+2}(j-k-1)\prod_{i\leqslant k}(2n-i-k)\prod_{j\geqslant k+2}(2n-j-k-1).\]
	The terms in $E(\lambda_{k+1})$ are
	\[\prod_{i\leqslant k}(k+1-i)\prod_{j\geqslant k+2}(j-k)\prod_{i\leqslant k}(2n-i-k+1)\prod_{j\geqslant k+2}(2n-j-k).\]
	Compare them, we get
	\[\frac{E(\lambda_k)}{E(\lambda_{k+1})}=\frac{(k+1)\times 1\times (2n-2k)\times 1}{1\times (n-k)\times (2n-k)\times 2}=\frac{k+1}{2n-k}<1.\]
	Similarly, we see
	\[\frac{E(\lambda_{n-1})}{E(\lambda_n)}=\frac{((n-1)+1)\times (2n-2(n-1))}{1\times (2n-(n-1))}=\frac{2n}{n+1}>1.\]
	From this, we only need to consider the case $\lambda_n$ $(n\geqslant 4)$ and $\lambda_3$ $(n\geqslant 4)$. 
	
	First consider $\lambda_n$. We have
	\begin{align*}
		\frac{E(\mu)}{E(\lambda_n)}=&\frac{\prod_{j=2}^n(j+1)}{\prod_{j=2}^n(j-1)}\times\frac{\prod_{j=2}^n(2n+1-j)\prod_{2\leqslant i<j}(2n-i-j)}{\prod_{1\leqslant i<j}(2n+2-i-j)}\\
		=&\frac{(n+1)n}{2}\times\frac{\prod_{j=2}^n(2n+1-j)\prod_{2\leqslant i<j}(2n-i-j)}{\prod_{0\leqslant i<j\leqslant n-1}(2n-i-j)}\\
		=&\frac{(n+1)n}{2}\times\frac{\prod_{j=2}^n(2n+1-j)\prod_{i=2}^{n-1}(n-i)}{\prod_{j=1}^{n-1}(2n-j)\prod_{j=2}^{n-1}(2n-1-j)}\\
		=&\frac{(n+1)n}{2}\times\frac{1\times\cdots\times(n-2)}{n\times\cdots\times(2n-3)}
	\end{align*}
	Since $n\geqslant 4$, we have
	\[\frac{(n+1)n}{2}\times\frac{1\times \cdots\times(n-2)}{n\times\cdots\times (2n-3)}\leqslant \frac{(n+1)n}{2}\times\frac{1\times 2}{n(n+1)}=1.\]
	The equality holds only when $n=4$, i.e. $\dim V(1,1,1,1)=\dim V(2,0,0,0)$.

	Next consider $\lambda_3$. When $n=4$, we have $\dim V(\lambda_3)>\dim V(\lambda_4)\geqslant \dim V(\mu)$. Thus assume $n\geqslant 5$. We have
	\[E(\lambda_3)=\prod_{\substack{i<j\leqslant 3\\4\leqslant i<j}}(j-i)\prod_{\substack{i\leqslant 3\\ j\geqslant 4}}(1+j-i)\prod_{i<j\leqslant 3}(2+2n-i-j)\prod_{4\leqslant i<j}(2n-i-j)\prod_{\substack{i\leqslant 3\\ j\geqslant 4}}(1+2n-i-j).\]
	Then
	\begin{align*}
		\frac{E(\mu)}{E(\lambda_3)}=&\frac{\prod_{j\geqslant 2}(j+1)\prod_{j\geqslant 3}(j-2)\prod_{j\geqslant 4}(j-3)}{\prod_{i<j\leqslant 3}(j-i)\prod_{j\geqslant 4}j(j-1)(j-2)}\times\\
		&\frac{\prod_{j\geqslant 2}(2n+1-j)\prod_{j\geqslant 3}(2n-2-j)\prod_{j\geqslant 4}(2n-3-j)}{\prod_{i<j\leqslant 3}(2+2n-i-j)\prod_{j\geqslant 4}(2n-j)(2n-1-j)(2n-2-j)}\\
		=&\frac{3(n+1)}{(n-1)(n-2)}\times\frac{n-2}{2n}<1
	\end{align*}
\end{proof}

\begin{lem}\label{second type}
	When $k\geqslant 1$, we have
	\[\dim V(2,\underbrace{1,\ldots,1}_k,0,\ldots,0)>\dim V(2,0,\ldots,0).\]
\end{lem}
\begin{proof}
	Denote $\mu=(2,0\ldots,0)$, $\lambda_k=(2,\underbrace{1,\ldots,1}_k,0,\ldots,0)$, $k\leqslant n-1$.
	
	We prove that if $k\leqslant n-3$, then $\dim V(\lambda_{k+1})>\dim V(\lambda_k)$. Only need to compare terms in $E(\lambda)$ involving index $k+2$. The terms in $E(\lambda_k)$ are
	\[(k+3)\prod_{i=2}^{k+1}(k+3-i)\prod_{j\geqslant k+3}(j-k-2)\cdot(2n-k-1)\prod_{i=2}^{k+1}(2n-i-k-1)\prod_{j\geqslant k+3}(2n-k-j-2).\]
	The terms in $E(\lambda_{k+1})$ are
	\[(k+2)\prod_{i=2}^{k+1}(k+2-i)\prod_{j\geqslant k+3}(j-k-1)\cdot (2n-k)\prod_{i=2}^{k+1}(2n-i-k)\prod_{j\geqslant k+3}(2n-k-j-1).\]
	Thus
	\begin{align*}
		\frac{E(\lambda_k)}{E(\lambda_{k+1})}=&\frac{(k+3)\times(k+1)\times 1\times (2n-k-1)\times (2n-2k-2)\times 1}{(k+2)\times 1\times (n-k-1)\times (2n-k)\times (2n-k-2)\times 2}\\
		=&\frac{(k+1)(k+3)}{k+2}\times\frac{2n-k-1}{(2n-k-2)(2n-k)}<1
	\end{align*}
	Similarly,
	\[\frac{E(\lambda_{n-2})}{E(\lambda_{n-1})}=\frac{2(n+1)^2(n-1)}{n^2(n+2)}>1.\]
	
	To conclude, only need to consider the case $\lambda_{n-1}$ and $\lambda_1$.
	
	First consider $\lambda_{n-1}$. We have
	\begin{align*}
		\frac{E(\mu)}{E(\lambda_{n-1})}=&\frac{\prod_{j=2}^n(j+1)}{\prod_{j=2}^nj}\times\frac{\prod_{j=2}^n(2n+1-j)\prod_{2\leqslant i<j}(2n-i-j)}{\prod_{j=2}^n(2n+2-j)\prod_{2\leqslant i<j}(2n+2-i-j)}\\
		=&\frac{(n+1)^2}{4n}\times\frac{\prod_{i=2}^{n-1}(n-i)}{\prod_{j=2}^{n-1}(2n-1-j)}\\
		=&\frac{(n+1)^2}{4n}\times\frac{1\times\cdots\times(n-2)}{n\times\cdots\times(2n-3)}
	\end{align*}
	When $n=3$, this quantity is $4/9<1$. When $n\geqslant 4$, we have
	\[\frac{(n+1)^2}{4n}\times\frac{1\times\cdots\times(n-2)}{n\times\cdots\times(2n-3)}\leqslant\frac{(n+1)^2}{4n}\times\frac{1\times 2}{n(n+1)}<1.\]
	
	Next consider $\lambda_1$. When $n=3$, we have $\dim V(\lambda_1)>\dim V(\lambda_2)>\dim V(\mu)$. Thus assume $n\geqslant 4$.  We have
	\[E(\lambda_1)=2\prod_{j=3}^n(j+1)\prod_{j=3}^n(j-1)\prod_{3\leqslant i<j}(j-i)\cdot(2n)\prod_{j=3}^n(2n-j+1)\prod_{j=3}^n(2n-j-1)\prod_{3\leqslant i<j}(2n-i-j).\]
	Then
	\begin{align*}
		\frac{E(\mu)}{E(\lambda_1)}=&\frac{\prod_{j=2}^n(j+1)\prod_{j\geqslant 3}(j-2)\prod_{j\geqslant 2}(2n+1-j)\prod_{j\geqslant 3}(2n-2-j)}{2\prod_{j\geqslant 3}(j+1)\prod_{j\geqslant 3}(j-1)\cdot(2n)\prod_{j\geqslant 3}(2n+1-j)\prod_{j\geqslant 3}(2n-1-j)}\\
		=&\frac{3(2n-1)}{8n(n-1)}<1
	\end{align*}
\end{proof}

\begin{lem}
    If $V(\mu)$ is irreducible representation of $\SO_{2n,J}(\BC)$ such that $f(H)$ (see equation (\ref{isom1})) fixes a non-zero vector, then $\sum\mu_i$ is even.
\end{lem}
\begin{proof}
    Suppose $v\in V(\mu)$ is fixed by $f(H)$. Decompose $v$ into weight spaces, we see each component is fixed by $f(H)\cap T$. Thus there is a weight $\mu'$ that is trivial on $f(H)\cap T$. Since $-1\in f(H)\cap T$, we see $\mu'(-1)=(-1)^{\sum\mu'_i}=1$, hence $\sum\mu'_i$ is even.

    Now $\mu-\mu'$ is integer sum of roots, while each root $\alpha$ also satisfies $\sum\alpha_i$ is even, hence $\sum\mu_i$ is even.
\end{proof}

\begin{prop}\label{type1-condition}
    Let $\mu\in\BZ^n$ be a dominant weight with $\sum\mu_i$ even. If $\dim V(\mu)\leqslant\dim V(2,0,\ldots,0)$, then $\mu$ is $(2,0,\ldots,0)$ or of the following form
    \[0,\quad (1,1,0,\ldots,0),\quad (1,1,1,\pm 1).\]
\end{prop}
\begin{proof}
    If $\lambda$ is dominant and $\lambda\neq 0$, then $\dim V(\mu+\lambda)>\dim V(\mu)$. Thus if $\mu_1-\mu_2\geqslant 2$, we are done. 
	
    \textbf{Case (i)}: $\mu_1-\mu_2=1$. Reduce to $\mu=(2,\underbrace{1,\ldots,1}_k,0,\ldots,0)$, $1\leqslant k\leqslant n-1$. The claim follows from Lemma \ref{second type}.
	
    \textbf{Case (ii)}: $\mu_1=\mu_2$. Reduce to $\mu=(\underbrace{1,\ldots,1}_k,0,\ldots,0)$, $1\leqslant k\leqslant n$. When $k\geqslant 3$ and $n\geqslant 4$, the claim follows from Lemma \ref{first type}.
	
    When $n=3$, we have $(1,0,0)$, $(1,1,0)$ and $(1,1,1)$. Considering the parity condition, we get the claim.
\end{proof}

\begin{lem}\label{even-exception}
    Consider $n=4$. Then $V(1,1,1,1)\oplus V(1,1,1,-1)$ is isomorphic to $\Lambda^4\BC^8$ (the wedge product) and $H$ does not fix a non-zero vector in $\Lambda^4\BC^8$.
\end{lem}
\begin{proof}
    The first claim follows from, for example \cite[Theorem 19.2, p.287]{FH91}. Consider the second claim. Use natural action of $\SO_8(\BR)$ on $\Lambda^4\BC^8$. Denote $e_i$ the standard basis of $\BC^8$. It is clear $e_i\wedge e_j\wedge e_k\wedge e_l$, $i<j<k<l$ are eigenvectors of $H$, but none of them is fixed by $H$.
\end{proof}

\begin{proof}[Proof of Proposition \ref{orth-flag2} for $\SO_{2n}(\BR)$, $n\geqslant 3$]
    Let $V$ be an irreducible real representation of $\SO_{2n}(\BR)$ such that $\dim V$ is no greater than isospectral model and $H$ fixes some non-zero vector in $V$. There is an irreducible summand $V(\mu)\subseteq V_\BC$ such that $H$ fixes some non-zero vector in $V(\mu)$. 
    
    Since $\dim V(\mu)\leqslant \dim V\leqslant\dim V(2,0,\ldots,0)$, Proposition \ref{type1-condition} implies $\mu$ should be $(2,0,\ldots,0)$ or $0$, $(1,1,0,\ldots,0)$, $(1,1,1,\pm 1)$. The cases $(2,0,\ldots,0)$ and $0$ are what we want to prove. The case $(1,1,1,\pm 1)$ is ruled out by above Lemma \ref{even-exception}. The case $(1,1,0,\ldots,0)$ corresponds to adjoint representation, which is ruled out by Lemma \ref{adj-no-fix}.
\end{proof}

\subsection{Case of $\SO_4(\BR)$}

The only highest weight representations $V(\mu_1,\mu_2)$ of $\SO_4(\BC)$ whose dimension is no greater than $\dim (2,0)=9$ and satisfies $2\mid \mu_1+\mu_2$ are $V(2,0)$ and
\[\dim V(0,0)=1,\quad \dim V(1,\pm1)=3,\quad\dim V(2,\pm 2)=5,\quad \dim V(3,\pm 3)=7,\quad\dim V(4,\pm 4)=9.\]

Among them, if we exclude trivial representation and those without $H$-fixed vectors, we are left with
\begin{equation}\label{rep-list}
\dim V(2,\pm 2)=5,\quad \dim V(3,\pm 3)=7,\quad\dim V(4,\pm 4)=9.
\end{equation}
This is because $V(1,1)\oplus V(1,-1)\cong\mathfrak{so}_4(\BC)$ and Lemma \ref{adj-no-fix} says the adjoint representation does not have $H$-fixed vector.

\begin{lem}
    Theorem \ref{orth-flag-min} holds for $\SO_4(\BR)/\RS(\RO_3(\BR)\times \RO_1(\BR))$, 
\end{lem}
\begin{proof}
    Let $V$ be an equivariant embedding of $\SO_4(\BR)/\RS(\RO_3(\BR)\times \RO_1(\BR))$ with smallest possible dimension. Assume $V\neq V(2,0)$. Then Lemma \ref{min-dim-fix} implies irreducible sub-representations of $V_\BC$ lie in above list (\ref{rep-list}).
    
    Since $\RS(\RO_3(\BR)\times \RO_1(\BR))$ contains a subgroup $\SO_3(\BR)$, each irreducible sub-representation of $V_\BC$ contains non-zero vector fixed by $\SO_3(\BR)$. We know $\SO_4(\BR)/\SO_3(\BR)\cong S^3$, where $S^3\incl\BR^4$ is the unit sphere. Equivariant embeddings of $S^3$ correspond to representations appearing in $L^2(S^3)$, which is spanned by homogeneous harmonic polynomials. In particular, they appear as sub-representations of symmetric tensor of standard representation $\BR^4$, i.e. in $\Sym^k(\BR^4)$, $k\geqslant 1$. Thus these representations coorespond to $V(k,0)$, $k\geqslant 1$, which does not appear in the list (\ref{rep-list}).
\end{proof}

Now consider equivariant embeddings of remaining cases, i.e. the required stabilizer is one of the following
\[\RS(\RO_2(\BR)\times\RO_2(\BR)),\quad\RS(\RO_2(\BR)\times\RO_1(\BR)\times\RO_1(\BR)),\quad \RS(\RO_1(\BR)\times\RO_1(\BR)\times\RO_1(\BR)\times\RO_1(\BR)).\]
Note that the Lie algebras of them are abelian.

\begin{lem}
    Theorem \ref{orth-flag-min} holds for the remaining cases. 
\end{lem}
\begin{proof}
    Let $V$ be an equivariant embedding with smallest possible dimension. Assume $V\neq V(2,0)$. Then Lemma \ref{min-dim-fix} implies irreducible sub-representations of $V_\BC$ lie in the list (\ref{rep-list}).

    We know $\SO_4(\BR)=\SU_2(\BC)\cdot\SU_2(\BC)$ is product of simple normal subgroups. The representation in the list (\ref{rep-list}) appear at most once for dimension reason. However, each such representation is a representation of one summand $\SU_2(\BC)$ of $\SO_4(\BR)$, hence the other summand $\SU_2(\BC)$ lies in the stabilizer. However in the remaining cases, the required stabilizer has abelian Lie algebra. If $\SU_2(\BC)$ lies in the stabilizer, then its Lie algebra should be abelian, reaching a contradiction.
\end{proof}

Combine above two Lemmas, we get

\begin{cor}\label{SO4-thm}
    Theorem \ref{orth-flag-min} holds for $\SO_4(\BR)$.
\end{cor}

\subsection{Case of $\SO_{2n+1}$}

Denote
\[S=\frac{1}{\sqrt 2}\begin{pmatrix}I_n&-iI_n\\I_n&iI_n\end{pmatrix},\]
For the matrix 
\[J=\begin{pmatrix} 1&0&0\\0&0&I_n\\0&I_n&0\end{pmatrix},\]
denote
\[\SO_{2n+1,J}(\BC)=\{A\in\GL_{2n+1}(\BC):A^tJA=J,\ \det A=1\}.\]
It is easily verified that
\begin{equation}
    f:\SO_{2n+1}(\BC)\to\SO_{2n+1,J}(\BC):A\mapsto\begin{pmatrix}1&0\\0&S\end{pmatrix}A\begin{pmatrix}1&0\\0&S\end{pmatrix}^{-1}
\end{equation}
is an isomorphism of algebraic groups.

Consider $\SO_{2n+1}$. Assume $n\geqslant 2$. Maximal weights are $\mu=(\mu_1,\dots,\mu_n)$ with $\mu_1\geqslant \mu_2\geqslant \cdots\geqslant \mu_n\geqslant 0$. Then 
\begin{equation}\label{dimension of representation of SO(2n+1)}
\dim V(\mu)=\prod_{i<j}\frac{\mu_i-\mu_j+j-i}{j-i}\prod_{i\leqslant j}\frac{\mu_i+\mu_j+2n+1-i-j}{2n+1-i-j}.\end{equation}

\begin{enumerate}
		\item $\mu=(2,0,\ldots,0)$.
		\begin{equation}\label{case 1}
		    \dim V(\mu)=n(2n+3).
		\end{equation}
		\item $\mu=(\underbrace{1,\ldots,1}_k,0,\ldots,0)$ for $k<n$.
		\begin{equation}\label{case 2}
		    \begin{aligned}
		    \dim V(\mu)&=\prod_{1\leqslant i\leqslant k}\dfrac{n+1-i}{k+1-i}\prod_{1\leqslant i\leqslant k}\dfrac{2n+1-i-k}{n+1-i}\prod_{1\leqslant i\leqslant j\leqslant k}\dfrac{2n+3-i-j}{2n+1-i-j}\\
            &=\prod_{1\leqslant i\leqslant k}\dfrac{2n+1-i-k}{k+1-i}\prod_{1\leqslant i\leqslant j\leqslant k}\dfrac{2n+3-i-j}{2n+1-i-j}.
		\end{aligned}
		\end{equation}

        Specially, when $k=1$, 
        \begin{equation}\label{case 2.1}
            \dim(V(1,0,\dots,0))=2n+1.
        \end{equation}
        when $k=2$, 
        \begin{equation}\label{case 2.2}
            \dim(V(1,1,0,\dots,0))=n(2n+1).
        \end{equation}
        when $k=3$,
        \begin{equation}\label{case 2.3}
            \dim(V(1,1,1,,0,\dots,0))=\dfrac{n(2n+1)(2n-1)}{3}.
        \end{equation}
       \item $\mu=(1,1,\dots,1)$, 
    \begin{equation}\label{case 3}
         \dim V(1,1,\dots,1)=\prod_{1\leqslant i\leqslant j\leqslant n}\dfrac{2n+3-i-j}{2n+1-i-j}.
    \end{equation}
   
\end{enumerate}

\begin{lem}\label{first type1}
    When $k\geqslant 3$, we have
	\[\dim V(\underbrace{1,\ldots,1}_k,0,\ldots,0)>\dim V(2,0,\ldots,0).\]
    
\end{lem}
\begin{proof}
	Denote $\mu=(2,0,\ldots,0)$ and $\lambda_k=(\underbrace{1,\ldots,1}_k,0,\ldots,0)$. First we prove that when $k\leqslant n-1$, we have $\dim V(\lambda_{k+1})>\dim V(\lambda_k)$. 

Note that by equation ~\eqref{case 1}, when $k<n-1$, 
    \begin{equation}\label{quotient of V(mu)}
    \begin{aligned}
        \dfrac{\dim(V(\lambda_{k+1}))}{\dim(V(\lambda_k))}&=\dfrac{2(n-k)(2n-1-2k)}{(2n-k)(k+1)}\cdot \dfrac{(2n+1-k)(2n-k)}{2(n-k)(2n-2k-1)}\\
        &=\dfrac{2n+1-k}{k+1}\geqslant \dfrac{k+2}{k+1}>1
    \end{aligned}
    \end{equation}
    And by equations ~\eqref{case 2} and ~\eqref{case 3},
    \[\begin{aligned}
    \dfrac{\dim V(\lambda_n)}{\dim V(\lambda_{n-1})}&=\dfrac{n+1}{n}>1
    \end{aligned}
    \]


Thus to finish the proof of the lemma, we only need to show that for $n\geqslant 3$,
\begin{equation}
    \dim(V(\lambda_3))> \dim(V(\mu))
\end{equation}
This is clear from equations ~\eqref{case 1} and ~\eqref{case 2.3}, since when $n\geq 3$,
\[
\dfrac{n(2n+1)(2n-1)}{3}-n(2n+3)=n\cdot \dfrac{4n^2-6n-10}{3}>0
\]

\end{proof}

\begin{remark}
    We have $\dim V(1,1,0\ldots,0)<\dim V(2,0,\ldots,0)$ and $\dim V(1,0\ldots,0)<\dim V(2,0,\ldots,0)$.
\end{remark}



\begin{lem}\label{second type1}
    Assume that $n\geqslant 2$. Let $\mu^k:=(2,\underbrace{1,\dots,1}_{k-1},0,\dots,0)$ for $1\leqslant k\leqslant n$. Then for $1\leqslant k\leqslant n-1$
    \begin{equation}
        \dim V(\mu^{k+1})>\dim V(\mu^{k})
    \end{equation}
\end{lem}
\begin{proof}
    By equation ~\eqref{dimension of representation of SO(2n+1)},  we get that for $n>2$ and $1<k<n-1$,
    \[
    \begin{aligned}
        \dfrac{\dim V(\mu^{k+1})}{\dim V(\mu^k)}
        &=\dfrac{k+1}{k+2}\cdot \dfrac{n-k}{k}\cdot \dfrac{2n-2k+1}{2n-2k-1}\cdot \dfrac{2n-2k-1}{n-k}\cdot \dfrac{2n-k}{2n-2k+1} \cdot \dfrac{2n-k+2}{2n-k+1}\\
        &=\dfrac{(k+1)(2n-k)(2n-k+2)}{k(k+2)(2n-k+1)}
    \end{aligned} 
    \]
    Since $k<n-1$, then $2n-k>k+2$, so 
    \[
    \dfrac{\dim V(\mu^{k+1})}{\dim V(\mu^k)}>\dfrac{k+1}{k}\cdot \dfrac{2n-k+2}{2n-k+1}>1
    \]
    Thus the lemma holds for $n>2$ and $1<k<n-1$. For $k=1$, $n>2$, by direct calculation, 
    \[
    \dim V(2,1,0,\dots,0)=\dfrac{(2n+3)(2n+1)(2n-1)}{3}>n(2n+3)=\dim V(2,0,\dots,0)
    \]
    where the last equality is from equation ~\eqref{case 1}.

    At last,  we can directly get that
    \[
    \dim V(2,1)=35>14=\dim V(2,0).
    \]
\end{proof}

\begin{prop}\label{type2-condition}
	Let $\mu$ be a dominant weight and $n\geq 2$. Consider the cases for $SO_{2n+1}(\mathbb{R})$ If $\dim V(\mu)<\dim V(2,0,\ldots,0)$, then $\mu$ is of the following form
	\[0,\quad (1,0,\ldots,0),\quad (1,1,0,\ldots,0).\]
\end{prop}
\begin{proof}
	If $\lambda$ is dominant and $\lambda\neq 0$, then $\dim V(\mu+\lambda)>\dim V(\mu)$ by equation ~\eqref{dimension of representation of SO(2n+1)} . Thus if $\mu_1-\mu_2\geqslant 2$, we are done. 
	
	\textbf{Case (i)}: $\mu_1-\mu_2=1$. Reduce to $\mu=(2,\underbrace{1,\ldots,1}_k,0,\ldots,0)$, $1\leqslant k\leqslant n-1$. When $n\geqslant 2$, the claim follows from Lemma \ref{second type1}. 
	
	\textbf{Case (ii)}: $\mu_1=\mu_2$. Reduce to $\mu=(\underbrace{1,\ldots,1}_k,0,\ldots,0)$, $1\leqslant k\leqslant n$. When $k\geqslant 3$, the claim follows from Lemma \ref{first type1}.
\end{proof}

Now onsider $SO_3(\mathbb{R})$. The irreducible representations of $SO_3(\BR)$ are $(H_k,\rho_k)$ for $k\geq 1$, which are defined as follows: 
\[
H_k:=\{f\in \BR[x_1,x_2,x_3]|\Delta f=0, f \text{ is homogenous of degree }k\}\bigcup \{0\},
\]
\[
(\rho_k(g)(f))(x):=f(g^{-1}x), g\in SO_3(\mathbb{R}), f\in H_k,
\]
where $\Delta:=\frac{\partial^2 }{\partial x_1^2 }+\frac{\partial^2 }{\partial x_2^2 }+\frac{\partial^2 }{\partial x_3^2 }$ is the usual Laplacian operator on $\mathbb{R}^3$. 
Let $H=\{I_3,D_1:=\diag(1,-1,-1),D_2:=\diag(-1,1,-1),D_3:=\diag(-1,-1,1)\}$ be the subgroup of $SO_3(\mathbb{R})$ and $$P_k:=\{f\in \BR[x_1,x_2,x_3]| f \text{ is homogenous of degree }k\}\bigcup\{0\}$$. We extend the representation $(H_k,\rho_k)$ to $P_k$ in a natural way and denote the extended one also by $\rho_k$. Namely, 
\[
(\rho_k(g)(f))(x):=f(g^{-1}x), g\in SO_3(\mathbb{R}), f\in P_k,
\]

\begin{lem}\label{fixed point of harmonic polynomials}
    If there exists a nonzero polynomial $f\in H_k$ such that $f$ is fixed under $H$, then $k$ is even.
\end{lem}
\begin{proof}
    It is clear that $M_k:=\{x_1^{d_1}x_2^{d_2}x_{3}^{d_3}|d_1,d_2,d_3\in\mathbb{N},d_1+d_2+d_3=k\}$ is a bais of $P_k$ and the vectors in $M_k$ is the eigenvectors of $H$.  So if there exists a nonzero polynomial $f\in P_k$ such that $f$ is fixed under $H$, then each monomial in $M_k$ should be fixed by $H$.
    
    If $k$ is odd, then for any $p:=x_1^{d_1}x_2^{d_2}x_{3}^{d_3}\in P_k$, there exists $i\in\{1,2,3\}$ such that $d_i$ is odd and $d_l$ is even for $l\neq i$. This means that for $l\neq i$, $\rho_k(D_l)(p)=-p$. Thus there is no nonzero fixed polynomial in $P_k$. Since $(H_k,\rho_k)$ is a subrepresentation of $(P_k,\rho_k)$, then there doesn't exist a nonzero polynomial $f\in H_k$ such that $f$ is fixed under $H$, and thus the lemma holds.
\end{proof}

\begin{lem}\label{adjoint representation as harmonic polynomials}
     $(H_2,\rho_2)$ is isomorphic to the isospectral model for $SO_{3}(\mathbb{R})$.
\end{lem}
\begin{proof}
    Note that $$P_2=\{a_{1}x_1^2+a_2x_2^2+a_3x_3^2+a_4x_1x_2+a_5x_1x_3+a_6x_2x_3|a_i\in \BR\}.$$
    So for $f=a_{1}x_1^2+a_2x_2^2+a_3x_3^2+2a_4x_1x_2+2a_5x_1x_3+2a_6x_2x_3\in P_2$, then $\Delta f=2a_1+2a_2+2a_3$. Thus 
    \begin{equation}\label{structure of H_2}
        H_2=\{a_{1}x_1^2+a_2x_2^2+a_3x_3^2+a_4x_1x_2+a_5x_1x_3+a_6x_2x_3|a_i\in \BR,a_1+a_2+a_3=0\}.
    \end{equation}
    Let $\phi: H_2\longrightarrow \Sym_0(3,\mathbb{R})$ defined by $$\phi(a_{1}x_1^2+a_2x_2^2+a_3x_3^2+a_4x_1x_2+a_5x_1x_3+a_6x_2x_3)=\begin{pmatrix}
        a_1&a_4&a_5\\
        a_4&a_2&a_6\\
        a_5&a_6&a_3
    \end{pmatrix}.$$
    This well-defined since equation ~\eqref{structure of H_2} holds. 
    Note that for every $f\in H_2$, $f=\begin{pmatrix}
        x_1&x_2&x_3
    \end{pmatrix}\phi(f)\begin{pmatrix}
        x_1\\x_2\\x_3
    \end{pmatrix}.$
    Then by defintion, it's easy to see that $\phi$ is a $G-$isomorphism between $H_2$ and the isospectral model.
\end{proof}
\begin{proof}[Proof of Proposition \ref{orth-flag2} for $\SO_{2n+1}(\BR)$, $n\geqslant 1$]

 For $n=1$, Lemma ~\ref{fixed point of harmonic polynomials} and Lemma~\ref{adjoint representation as harmonic polynomials} implies the proposition.

 Now let $n\geqslant 2$ and $V$ be an irreducible real representation of $\SO_{2n+1}(\BR)$ such that $\dim V$ is no greater than isospectral model and $H$ fixes some non-zero vector in $V$. If $V(\mu)\subseteq V_\BC$ is an irreducible summand, then $H$ fixes some non-zero vector in $V(\mu)$ as well. By above discussion, in particular Proposition \ref{type2-condition} implies $\mu$ is of the form $0$, $(1,1,0,\ldots,0)$, $(1,0,\dots,0)$.

 The first one is excluded by Lemma \ref{natural representation-no-fix} since $V(1,0,\dots,0)$ is just the natuaral represention of $SO_{2n+1}(\BR)$. The second is ruled out by the lemma \ref{adj-no-fix} since $V(1,1,\dots,0)$ is just the adjoint representation of $SO_{2n+1}(\BR)$
\end{proof}



\section{Equivariant embedding of complex flag manifold}
The definition ~\ref{real flag} of real flag manifold can be natuarally extended to the complex case, that is, replacing all the real vector spaces by the complex one. In this section, we always assume that $V$ is a complex vector space and denote the corresponding complex flag manifold of type $(k_1,\dots,k_r)$ by $\CFlag(k_1,\dots,k_r,V)$

\begin{lem}\label{lemma: stablizer of complex flag manifold1}
	Fix a Hermitian inner product on $V$. Then the stabilizer in $U(V)$ of a flag $W_1\subset\cdots\subset W_r$ in $\CFlag(k_1,\ldots,k_r,V)$ is $\RU(U_1)\times\cdots\times\RU(U_{r+1})$. Here $U_i$ is the orthogonal complement of $W_{i-1}$ in $W_i$, and we identify $W_0=0$, $W_{r+1}=V$.
	
	If we take stabilizer in $\SU(V)$, then we get $\RS(\RU(U_1)\times\cdots\times\RU(U_{r+1}))$. Here the outer $S$ means the total determinant is $1$.
\end{lem}

We shall prove the following

\noindent\textbf{Theorem \ref{uni-min-equi}.} {\it
    The equivariant embedding of $\SU_n/\RS(\RU_{n_1}\times\cdots\times\RU_{n_r})$, $n_1+\cdots+n_r=n$, $n_i\geqslant 1$, $n\geqslant 2$ with minimal dimension is isospectral model.}

Denote $H\subseteq\SU_n$ the subgroup of diagonal matrices. It is clear that for any partition $(n_1,\ldots,n_r)$ of $n$, $H\subseteq\RS(\RU_{n_1}\times\cdots\times\RU_{n_r})$. Similar to Lemma \ref{min-dim-fix}, if $V$ reaches minimal dimension, then each irreducible summand of $V_\BC$ contains a non-zero $H$-fixed vector. We shall prove the following Proposition.

\begin{prop}\label{H-fix-min}
    Let $V$ be an irreducible complex representation of $\RU_n$. If $\dim V$ is no greater than isospectral model and contains non-zero $H$-fixed vector, then $V$ is either trivial representation or isopsectral model.
\end{prop}

It is easy to deduce Theorem \ref{uni-min-equi} from above Proposition.

\subsection{Computation of dimension}

It is known that $\SU_n\incl\SL_n(\BC)$ is the complexification, hence every complex representation of $\SU_n$ extends to a complex representation of $\SL_n(\BC)$. Take maximal torus $T$ of $\SL_n(\BC)$ to be diagonal matrices. Denote $e_i\in\BR^n$ standard basis vector with $1$ at $i$-th coordinate. We view $e_i$ as character on diagonal matrices that is only non-trivial at $(i,i)$-entry. The characters of $T$ then lie in the subspace $E=\{(x_1,\ldots,x_n)\mid \sum x_i=0\}$. The roots are $\Phi=\{e_i-e_j\mid i\neq j\}$. Choose positive roots to be $\Phi^+=\{e_i-e_j\mid i<j\}$. For each $\mu=(\mu_1,\ldots,\mu_{n-1})\in\BZ^{n-1}$, the vector
\[\Big(\mu_1-\frac{\sum\mu_i}{n},\ldots,\mu_{n-1}-\frac{\sum\mu_i}{n},-\frac{\sum\mu_i}{n}\Big)\]
is a weight, and it is dominant if and only if $\mu_1\geqslant\mu_2\geqslant\cdots\geqslant\mu_{n-1}\geqslant 0$. Conversely, all weights are of above form. We shall write $\mu\in\BZ^{n-1}$ to mean a weight identified as above. It is clear $\mu$ acts on $T$ via
\[\mu(\diag(a_1,\ldots,a_n))=\prod_{i=1}^{n-1}a_i^{\mu_i}.\]

By Weyl character formula, denote $V(\mu)$, $\mu=(\mu_1,\ldots,\mu_{n-1})$ the irreducible highest weight $\lambda$ representation, then
\[\dim V(\mu)=\frac{\prod_{i<j<n}(\mu_i-\mu_j+j-i)\prod_{i<n}(\mu_i+n-i)}{\prod_{i<j}(j-i)}\]

The isospectral model corresponds to $\mu=(2,1,\ldots,1)$, which has dimension $(n+1)(n-1)$.

Recall $H=T\cap\SU_n$.

\begin{lem}\label{n-div}
    If $V(\mu)$ contains a non-zero $H$-fixed vector, then $n\mid\sum\mu_i$.
\end{lem}
\begin{proof}
    If $V(\mu)$ contains a non-zero $H$-fixed vector, then it contains a weight $\lambda$ that is trivial on $H$. Thus $\lambda=0$. Since each root $\alpha=(\alpha_1,\ldots,\alpha_{n-1})$ satisfies $n\mid\sum\alpha_i$, and $\mu-\lambda$ is $\BZ$-linear combintations of roots, the claim follows.
\end{proof}

\begin{lem}\label{uni-ineq1}
    Suppose $\mu=(\underbrace{a,\ldots,a}_u,\underbrace{b,\ldots,b}_v,\mu_{u+v+1},\ldots,\mu_{n-1})$ with $u,v\geqslant 1$, $u+v\leqslant n-2$, $a>b>\mu_{u+v+1}$. Then
    \[\dim V(\mu)>\dim V(2,1,\ldots,1).\]
\end{lem}
\begin{proof}
    It is clear $n\geqslant 4$. Consider item in dimension formula involving indices in $\{1,\ldots,u+v\}$. We see
    \begin{align*}
        &\dim V(\mu)\\
        \geqslant &\prod_{\substack{i\leqslant u\\ u<j\leqslant u+v}}\frac{1+j-i}{j-i}\prod_{\substack{i\leqslant u\\ u+v<j\leqslant n-1}}\frac{2+j-i}{j-i}\prod_{\substack{u<i\leqslant u+v\\ u+v<j\leqslant n-1}}\frac{1+j-i}{j-i}\prod_{i\leqslant u}\frac{2+n-i}{n-i}\prod_{u<i\leqslant u+v}\frac{1+n-i}{n-i}\\
        =&\frac{v+1}{u!(u+v+1)!}\cdot\big((n+1)\cdots(n-u+2)\big)\cdot\big(n\cdots(n-u-v+1)\big)\\
        =&(n+1)n\cdot\frac{n\cdots(n-u+2)}{u!}\cdot\frac{(n-1)\cdots(n-u-v+1)}{(u+v+1)!}\cdot(v+1)
    \end{align*}
    Here in the last line, $n\cdots(n-u+2)$ means $1$ if $u=1$.
    
    It is clear $n\cdots(n-u+2)\geqslant u!$. Since $u+v+1\leqslant n-1$, we see
    \[(n-1)\cdots(n-u-v+1)\geqslant \frac{(u+v+1)!}{v+1}.\]
    Thus
    \[\dim V(\mu)\geqslant (n+1)n>(n+1)(n-1)=\dim V(2,1,\ldots,1).\]
\end{proof}

\begin{lem}\label{uni-ineq2}
    Suppose $\mu=(\underbrace{a,\ldots,a}_u,b,\ldots,b)$ with $1\leqslant u\leqslant n-2$, $a>b\geqslant 1$. Then
    \[\dim V(\mu)\geqslant \dim V(2,1,\ldots,1).\]
    Equaility holds only when $\mu=(2,1,\ldots,1)$.
\end{lem}
\begin{proof}
    It is clear $n\geqslant 3$. We have
    \begin{align*}
        \dim V(\mu)=&\prod_{\substack{i\leqslant u\\ u<j\leqslant n-1}}\frac{a-b+j-i}{j-i}\prod_{i\leqslant u}\frac{a+n-i}{n-i}\prod_{u<i\leqslant n-1}\frac{b+n-i}{n-i}\\
        \geqslant &\prod_{\substack{i\leqslant u\\ u<j\leqslant n-1}}\frac{1+j-i}{j-i}\prod_{i\leqslant u}\frac{2+n-i}{n-i}\prod_{u<i\leqslant n-1}\frac{1+n-i}{n-i}\\
        =&\frac{(n+1)\cdots(n-u+2)}{u!}\cdot (n-u)
    \end{align*}
    Denote
    \[A=\frac{(n+1)\cdots(n-u+2)}{u!}\cdot (n-u).\]
    
    \textbf{Case }$u=1$. We see $A=(n+1)(n-1)\dim V(2,1,\ldots,1)$. However, if $a>2$ or $b>1$, then above inequalities become strict, hence equality holds here only when $a=2$, $b=1$.

    \textbf{Case }$u=2$. Since $n\geqslant 4$, we have
    \[A=\frac{(n+1)n(n-2)}{2}>(n+1)(n-1)=\dim V(2,1,\ldots,1).\]

    \textbf{Case }$u\geqslant 3$. Then
    \[A=(n+1)n\cdot\frac{(n-1)\cdots(n-u+2)}{u!}\cdot(n-u).\]
    Since $u\leqslant n-2$,
    \[\big((n-1)\cdots(n-u+2)\big)\cdot (n-u)\geqslant (n-2)\cdots(n-u)\geqslant u!.\]
    Thus
    \[A\geqslant (n+1)n>(n+1)(n-1)=\dim V(2,1,\ldots,1).\]
\end{proof}

\begin{lem}\label{uni-ineq3}
    Suppose $\mu=(\underbrace{a,\ldots,a}_u,0,\ldots,0)$ with $2\leqslant u\leqslant n-2$, $a\geqslant 2$. Then
    \[\dim V(\mu)>\dim V(2,1,\ldots,1).\]
\end{lem}
\begin{proof}
    It is clear $n\geqslant 4$. We have
    \[\dim V(\mu)=\prod_{\substack{i\leqslant u\\ u<j\leqslant n}}\frac{a+j-i}{j-i}\geqslant \prod_{\substack{i\leqslant u\\ u<j\leqslant n}}\frac{2+j-i}{j-i}=\frac{(n+1)\cdots(n-u+2)}{(u+1)!}\cdot\frac{n\cdots (n-u+1)}{u!}.\]
    Denote
    \[A=\frac{(n+1)\cdots(n-u+2)}{(u+1)!}\cdot\frac{n\cdots (n-u+1)}{u!}\]

    \textbf{Case }$n=4$. Then $u=2$, we see
    \[A=\frac{(n+1)n^2(n-1)}{12}>(n+1)(n-1)=\dim V(2,1,\ldots,1).\]

    \textbf{Case }$n\geqslant 5$. Write
    \[A=(n+1)n\cdot\frac{n\cdots(n-u+2)}{(u+1)!}\frac{(n-1)\cdots(n-u+1)}{u!}.\]
    It is clear $(n-1)\cdots (n-u+1)\geqslant u!$. Since $u+1\leqslant n-1$, we see
    \[(u+1)\cdots 4\leqslant (n-1)\cdots (n-u+2).\]
    Thus
    \[A\geqslant (n+1)n\cdot\frac{n}{3!}=\frac{(n+1)n^2}{6}>(n+1)(n-1)=\dim V(2,1,\ldots,1).\]
\end{proof}

\begin{lem}\label{uni-ineq4}
    Suppose $n\geqslant 3$, $\mu=(n,\ldots,n)$. Then
    \[\dim V(\mu)>\dim V(2,1,\ldots,1).\]
\end{lem}
\begin{proof}
    We have
    \[\dim V(\mu)=\prod_{i\leqslant n-1}\frac{2n-i}{n-i}.\]
    When $n=3$, we have
    \[\dim V(\mu)=10>8=\dim V(2,1).\]
    When $n\geqslant 4$, we have
    \begin{align*}
    \dim V(\mu)=&\frac{(2n-1)(2n-2)}{2}\cdot\frac{(2n-3)\cdots(n+1)}{(n-1)\cdots 3}\\
    \geqslant&(2n-1)(n-1)\\
    >&(n+1)(n-1)=\dim V(2,1,\ldots,1).
    \end{align*}
    
\end{proof}

We can now prove Proposition \ref{H-fix-min}.

\begin{proof}[Proof of Proposition \ref{H-fix-min}]
    By Lemma \ref{n-div}, we know $n\mid\sum\mu_i$. 
    
    \textbf{Case (i)}: Suppose $\mu$ has at least three stairs, i.e. $\mu_1>\mu_i>\mu_j$ for some $1<i<j$. Then we conclude from Lemma \ref{uni-ineq1}.

    \textbf{Case (ii)}: Suppose $\mu$ has two stairs, i.e. $\mu=(\underbrace{a,\ldots,a}_u,b,\ldots,b)$ with $1\leqslant u\leqslant n-2$, $a>b$. If $b\geqslant 1$, we conclude from Lemma \ref{uni-ineq2}. 
    
    Assume $b=0$. Then $a\geqslant 2$ since $n\mid ua$. If $u\geqslant 2$, then we conclude from Lemma \ref{uni-ineq3}. If $u=1$, then $a$ is divislbe by $n$. In particular, $a\geqslant 3$. Then
    \[\dim V(\mu)\geqslant \dim V(3,0,\ldots,0)=\frac{n(n+1)(n+2)}{6}>(n+1)(n-1)=\dim V(2,1,\ldots,1).\]

    \textbf{Case (iii)}: Suppose $\mu$ has only one stair, i.e. $\mu=(a,\ldots,a)$. Then $a$ is divisible by $n$. When $n=2$, $\mu=(a)$, and $2\mid a$, hence only when $a=2$ achieves minimal dimension.

    Assume $n\geqslant 3$. Then Lemma \ref{uni-ineq4} implies
    \[\dim V(\mu)\geqslant \dim V(n,\ldots,n)>\dim V(2,1,\ldots,1).\]
\end{proof}

\bibliographystyle{plain}
\bibliography{equiv_ref}
	
\end{document}